\documentclass{amsart}

\usepackage{amsthm, amssymb, amscd, amsmath, amsfonts, amsbsy}
\usepackage{graphicx}
\usepackage{hyperref}
\usepackage[a4paper, margin=1in]{geometry}
\usepackage{float}
\usepackage{xfrac}
\usepackage{tikz-cd}
\usepackage{mathtools}
\usepackage[percent]{overpic}	
\usepackage{caption, subcaption}

\newcommand{\R}{\ensuremath{\mathbb R}}
\newcommand{\N}{\ensuremath{\mathbb N}}

\newcommand{\unknot}{U}

\newcommand{\diagrams}{\mathcal{D}}

\newcommand{\links}{\mathcal{L}}
\newcommand{\rlinks}{\mathcal{\bar L}}

\newcommand{\olinks}{\mathcal{\hat L}}
\newcommand{\olinkss}[1]{\olinks_{#1}}

\newcommand{\basis}{\mathcal{B}}
\newcommand{\rbasis}{\mathcal{\bar B}}
\newcommand{\rbasiss}[1]{\rbasis_{#1}}
\newcommand{\obasis}{\mathcal{\hat B}}
\newcommand{\obasiss}[1]{\obasis_{#1}}

\newcommand{\hsm}{\mathcal{H}}
\newcommand{\rhsm}{\mathcal{\bar H}}
\newcommand{\rhsms}[1]{\rhsm_{#1}}
\newcommand{\ohsm}{\mathcal{\hat H}}
\newcommand{\ohsms}[1]{\ohsm_{#1}}

\newcommand{\submodule}{S}
\newcommand{\rsubmodule}{\bar S}

\newcommand\icon[1]{ \raisebox{-0.7em}{\includegraphics[page=#1]{figs}}}

\newcommand\sumpair[1]{ \raisebox{-1.1em}{\begin{overpic}[page=#1]{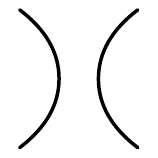}\put(5,27){\footnotesize{$L_1$}}\put(73,27){\footnotesize{$L_2$}}\end{overpic}}}

\newcommand\imga[1]{ \raisebox{-2.1em}{\begin{overpic}[page=22]{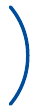}\put(20,55){\small{$#1$}}\end{overpic}}}
\newcommand\imgb[1]{\;\raisebox{-2.1em}{\begin{overpic}[page=24]{figures}\put(-8,46){\small{$#1$}}\end{overpic}}}
\newcommand\imgB[1]{\;\raisebox{-2.1em}{\begin{overpic}[page=53]{figures}\put(-8,46){\small{$#1$}}\end{overpic}}}
\newcommand\imgc[1]{ \raisebox{-2.1em}{\begin{overpic}[page=21]{figures}\put(4,44){\small{$#1$}}\end{overpic}}}
\newcommand\imgd[1]{ \raisebox{-2.1em}{\begin{overpic}[page=23]{figures}\put(4,44){\small{$#1$}}\end{overpic}}}

\newcommand\smgenth[1]{ \raisebox{-1.85em}{\begin{overpic}[page=14]{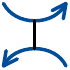}\put(15,61){\small{$#1$}}\end{overpic}}}
\newcommand\smgenthb[1]{ \raisebox{-1.85em}{\begin{overpic}[page=15]{icons}\put(-35,45){\small{$#1$}}\end{overpic}}}
\newcommand\smgenH[1]{ \raisebox{-1.85em}{\begin{overpic}[page=16]{icons}\put(39,44){\small{$#1$}}\end{overpic}}}
\newcommand\smgenHb[1]{ \raisebox{-1.85em}{\begin{overpic}[page=17]{icons}\put(40,44){\small{$#1$}}\end{overpic}}}

\newcommand\Q[1]{\raisebox{-2.3em}{\includegraphics[page=#1]{proof}}}
\newcommand\iconr[1]{\raisebox{-0.7em}{\includegraphics[page=#1]{figures}}}

\newcommand\Qtext[4]{\raisebox{-2.3em}{
\begin{overpic}[page=#1]{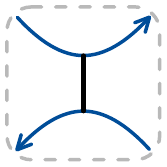}\put(#2,#3){\small{$#4$}}\end{overpic}}
}

\newcommand{\rei}[1]{\raisebox{0.65em}{\;\;$\xleftrightarrow{#1}$\;\;}}

\newcommand\icn[1]{ \raisebox{-0.7em}{\includegraphics[page=#1]{icons}}}

\newcommand\icm[1]{ \raisebox{-0.9em}{\includegraphics[page=#1]{icons}}}

\newcommand\icex[1]{ \raisebox{-0.9em}{\includegraphics[page=#1]{example}}}

\newcommand\smallicn[1]{ \raisebox{-0.55em}{\includegraphics[page=#1]{icons}}}

\newcommand\unknotfig{ \raisebox{-0.35em}{\includegraphics[page=12]{icons}}}
\newcommand\unlinkfig{ \raisebox{-0.35em}{\includegraphics[page=13]{icons}}}

\newcommand\rHSMknot[1]{\bigg[\!\!\raisebox{-0.25em}{#1}\bigg]_{\rbasis}}
\newcommand\HSMknot[1]{\bigg[\!\!\raisebox{-0.25em}{#1}\,\bigg]_{\basis}}
\newcommand\hpoly[1]{P\bigg(\!\!\!\raisebox{-0.3em}{#1}\bigg)}

\newcommand{\bTheta}{\bar\Theta}
\newcommand{\bH}{\bar{H}}

\newcommand{\Ti}{\Theta_i}
\newcommand{\bTi}{\bar\Theta_i}
\newcommand{\Hi}{H_i}
\newcommand{\bHi}{\bH_i}

\newcommand{\Tci}{\Theta_{c_i}}
\newcommand{\bTci}{\bar\Theta_{c_i}}
\newcommand{\Hci}{H_{c_i}}
\newcommand{\bHci}{\bH_{c_i}}

\newcommand{\smT}{\smallicn{19}}
\newcommand{\smH}{\smallicn{21}}

\newcommand{\smbT}{\smallicn{20}}
\newcommand{\smbH}{\smallicn{22}}

\newcommand{\Ttt}{\smT\smbT\smbT}
\newcommand{\Tth}{\smT\smbT\smbH}
\newcommand{\Thh}{\smT\smbH\smbH}
\newcommand{\ttH}{\smbT\smbT\smH}
\newcommand{\tHh}{\smbT\smH\smbH}
\newcommand{\Hhh}{\smH\smbH\smbH}

\newcommand{\Heq}{\stackrel{\eqref{eq:skein}}{=}}
\newcommand{\isoteq}{\stackrel{\text{isot.}}{=}}
\newcommand{\sumeq}{\stackrel{\text{\eqref{lemma:sum}}}{=}}
\newcommand{\eqeq}[1]{\stackrel{\text{\eqref{#1}}}{=}}

\newcommand{\tensor}{\otimes}

\theoremstyle{definition}

\newtheorem{example}{Example}[section]
\newtheorem{theorem}{Theorem}[section]
\newtheorem{lemma}{Lemma}[section]
\newtheorem{proposition}{Proposition}[section]
\newtheorem{remark}{Remark}[section]

\title{An invariant for colored bonded knots}
\author[B. Gabrov\v sek]{Bo\v{s}tjan Gabrov\v sek}
\address[Bo\v{s}tjan Gabrov\v sek]{University of Ljubljana, Faculty of Mechanical Engineering\\
        Aškerčeva 6, 1000 Ljubljana, Slovenia; University of Ljubljana, Faculty of Mathematics and Physics\\
        Jadranska 19, 1000 Ljubljana, Slovenia.}
        
\email[B.~Gabrov\v sek]{bostjan.gabrovsek@fs.uni-lj.si}
\date{\today}

\keywords{HOMFLYPT polynomial, skein modules, bonded knots, protein knots}

\subjclass{Primary 57M25; Secondary 05C15, 	92D99}

\begin{document}

\begin{abstract}
We equip a knot $K$ with a set of colored bonds, that is, colored intervals properly embedded into $\R^3 \setminus K$. Such a construction can be viewed as a structure that topologically models a closed protein chain including any type of bridges connecting the backbone residues. 
We introduce an invariant of such colored bonded knots that respects the HOMFLYPT relation, namely the HOMFLYPT skein module of colored bonded knots.
We show that the rigid version of the module is freely generated by colored $\Theta$-curves and handcuff links, while the non-rigid version is freely generated by the trivially embedded $\Theta$-curve. The latter module, however, does not provide information about the knottedness of the bonds. 


\end{abstract}

\maketitle

\section{Motivation}\label{sec:motivation}
In this paper we present a topological model, which models, as closely as possible, a closed protein backbone chain with bonds and compute a HOMFLYPT-type invariant to distinguish between such structures.

The existence of knotted proteins was first proposed in 1994~\cite{Mansfield1994} and discovered shortly after~ \cite{Liang1994}.
Today we know that knotted proteins are not uncommon as approximately 1\% of all entries in the protein data bank (PDB) are knotted~\cite{Virnau2006, Mishra2011}.
The function of protein knots is not yet completely understood, but it is hypothesised that knotedness increases thermal and kinetic stability of the molecule~\cite{Sulkowska2008}.

In order to study the topology of a protein backbone, we require that it forms a closed loop, which also has a natural orientation induced by the  amino and carboxyl termini of the protein.

The existence of an unambiguous closure method is still an open question, but common methods are of probabilistic nature (e.g.~\cite{Dabrowski-Tumanski2019, Jamroz2014, Sulkowska2012}), opposed to direct methods (e.g.~\cite{Virnau2006}). In the example in Figure~\ref{fig:trefoil} direct segment closure is used. 
In this paper we assume the proteins are closed by any of these methods, since closure methods are already thoroughly discussed in  \cite{Millett2013,Goundaroulis2017}. 


\begin{figure}[b]
\centering
\subcaptionbox{The YBEA methyltransferase from E. coli (PDB entry 1NS5).\label{fig:tref1}}
{\includegraphics[page=1]{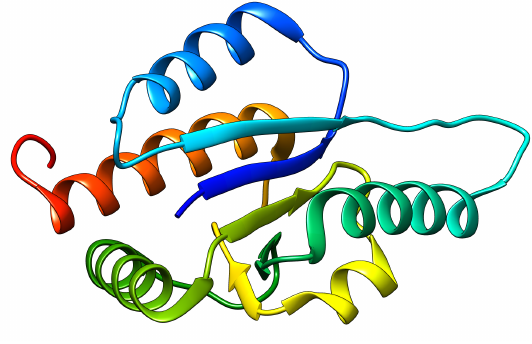}}
\qquad\quad
\subcaptionbox{The corresponding topology of the backbone (closed by a direct segment).\label{fig:tref2}}
{\includegraphics[page=2]{toxin}}
\caption{An example of a knotted protein and its backbone topology forming an oriented right-handed trefoil knot.}\label{fig:trefoil}
\end{figure}

There are recently proposed models that do not involve closing the backbone, for example, virtual knots~\cite{Alexander2017} and knottoids (open-knotted curves)~\cite{Turaev2012, Goundaroulis2017a, Goundaroulis2017, Gueguemcue2017, Gueguemcue2017a}.

The three-dimensional protein structure also consists of bonds tying parts of the peptide backbone. These bonds have both a structural and functional role and can be of several types: hydrogen bonds, hydrophobic interactions, salt bridges, disulphide brides, etc. From the perspective of knot theory it seems imperative that these bonds are encoded into the knotted topological structure itself.

Some of the existing literature that studies such bonds include \cite{Tian2017}, where the bond is replaced by a tangle, a similar approach is used in \cite{Goundaroulis2017}, where the underlying structures are (rigid-vertex) bonded planar knotoids.
Taking into consideration only one bond, the topology can be studied via $\Theta$-curves (graphs with two vertices and three parallel edges~\cite{Kauffman1993, Chlouveraki2019}): in \cite{Dabrowski-Tumanski2019} they identified 7 topologically inequivalent protein $\Theta$-curve structures and in \cite{ODonnol2018} $\Theta$-curve analysis was used to study knots appearing during DNA replication. 
In general, these models either transform the bonded structure to a classical knot or link by tangle replacement or view the structure as a spatial graph. 

The advantages of studying bonded knots, defined in the next section, are:
\begin{itemize}
\item bonded knots can contain any number of bonds,
\item the coloring function expresses different types of chemical bonds,
\item we differentiate between the bond and the backbone strand,
\item we take in to account the orientation of the protein.
\end{itemize}


This paper is organized as follows.
In Section~\ref{sec:bonded} we define non-rigid and rigid versions of colored bonded knots,
in Section~\ref{sec:hsm} we define the HOMFLYPT skein module of colored bonded knots, in Section~\ref{sec:rigid} we compute the HOMFLYPT skein module of rigid colored bonded knots, in Section~\ref{sec:non-rigid} we compute the HOMFLYPT skein module of non-rigid colored bonded knots, and in Section~\ref{sec:refined} we extend the coloring function of the rigid case and compute examples.

\section{Bonded knots}\label{sec:bonded}

A \emph{colored bonded knot} (or link) is a triple $(L, \mathbf{b}, \chi)$, where:
\begin{itemize}
\item $L$ is an oriented knot (or link) embedded in the three-sphere $S^3$ (or $\R^3$),
\item  $\mathbf{b} = \{b_1, b_2,\ldots,b_n\}$ is a set of pairwise disjoint \emph{bonds} (closed intervals) properly embedded into $S^3 \setminus L$,
\item $\chi:\mathbf{b} \rightarrow \N$ is a \emph{coloring function} on the bonds.
\end{itemize}

A \emph{bonded link diagram} is a regular projection of $L$ and the $b$'s to a plane with information of under/overcrossings and the coloring. We call a projection regular if it is cuspless and the only multiple points are finitely many transverse double points not coinciding with vertices (intersections between the link and the bonds), see Figure~\ref{fig:forb} for the list of forbidden positions.

\begin{figure}
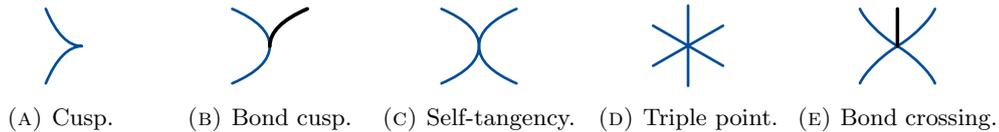

\centering
\subcaptionbox{Cusp.\label{fig:f1}}[7.5em]{\iconr{32}}
\subcaptionbox{Bond cusp.\label{fig:f2}}[7.5em]{\iconr{35}}
\subcaptionbox{Self-tangency.\label{fig:f3}}[7.5em]{\iconr{33}}
\subcaptionbox{Triple point.\label{fig:f4}}[7.5em]{\iconr{34}}
\subcaptionbox{Bond crossing.\label{fig:f5}}[7.5em]{\iconr{36}}
\caption{Forbidden diagram positions.}\label{fig:forb}
\end{figure}

\begin{figure}
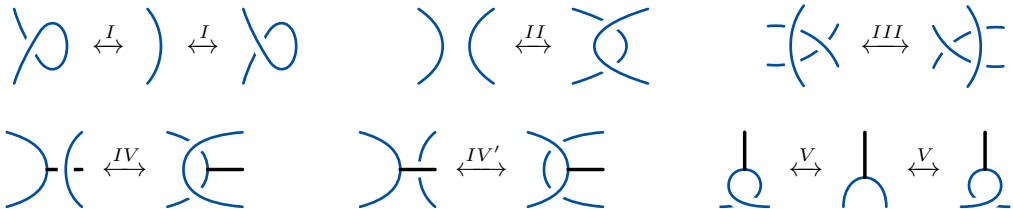

$$\iconr{49}\rei{I}\iconr{1}\rei{I}\iconr{2} \hspace{4em}
  \iconr{3}\rei{II}\iconr{4} \hspace{4em}
  \iconr{5}\rei{III}\iconr{6} $$

$$\iconr{7}\rei{IV}\iconr{8} \hspace{4em}
\iconr{9}\rei{IV'}\iconr{54} \hspace{4em}
\iconr{55} \rei{V}\iconr{56} \rei{V}\iconr{57} $$ 
  \caption{Reidemesiter moves for bonded links. Although not indicated in the figures, arcs in moves I, II, III and the free arcs in IV and IV' can be either link arcs or (colored) bonds.} \label{fig:reidemeister}
\end{figure}

The following theorem is a special case of the Reidemeister theorem for spatial graphs:

\begin{theorem}[special case of Theorem 2.1 in \cite{Kauffman1989}]
Two bonded links are ambient isotopic if and only if their diagrams are related by a finite sequence of (Reidemeister) moves I -- V.
\end{theorem}

In other words, ambient isotopy of bonded links is generated by moves I -- V.

We denote by $\links$ the set of all colored bonded links and by $\diagrams$ the set of all colored bonded link diagrams. 

Before attacking the case of bonded knots, we will construct a somewhat combinatorically simpler object, namely \emph{rigid colored bonded knots}, which arise if we replace the move V with the move RV in Figure~\ref{fig:rigid}.
More precisely, rigid colored bonded knots are equivalence classes $\rlinks = \diagrams/{\kern -0.02em\sim}$, where two diagrams $D_1, D_2 \in \diagrams$ are equivalent iff they are connected through planar isotopy and a finite sequence of moves I -- IV and RV.

\begin{figure}
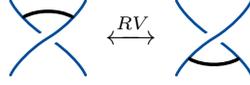

$$\iconr{43}\rei{RV}\iconr{44}$$ 
  \caption{A rigid-vertex version of move V.} \label{fig:rigid}
\end{figure}

As observed by Kauffman in~\cite{Kauffman1989}, rigid-vertex graphs are indeed easier to study, ``The problem of constructing good invariants in the general case is quite difficult,
due to the generation of arbitrary braiding at a vertex via the extended
move V.'' Non-rigidity, however, better reflect spatial isotopy and should be a model for bonded proteins, since the peptide chain has in general some degree of rotational freedom (see \cite{Tropp2012}).


\section{The HOMFLYPT skein module}\label{sec:hsm}

The HOMFLYPT polynomial~\cite{Freyd1985} can be regarded as one of the strongest polynomial knot  invariants as both the Alexander polynomial and Jones polynomial can be obtained from it by appropriate substitutions. In the case where the ambient space is not $S^3$, but any other connected and orientable 3-manifold, the polynomial may not be well-defined. Przytycki~\cite{Przytycki1991, Przytycki2006} and Turaev~\cite{Turaev1990} independently overcame this problem by introducing the concept of skein modules. In short, we obtain a skein module if we take the space of all finite linear combinations of links and on this space impose the skein relation through which the classical HOMFLYPT polynomial is defined.
The HOMFLYPT skein module, in particular, is hard to compute as it has been computed for only a handful of spaces~\cite{Turaev1990,Przytycki1992, Mroczkowski2004,Gabrovsek2014} (see also~\cite{Diamantis2016,Diamantis2016a,Diamantis2019,Mroczkowski2018} for algebraic aproaches and similar modules). The concept of a skein module has also been applied to other knotted structures, such as singular knots~\cite{Paris2013} (immersed circles in $\R^3$ with finitely many singular points). In the following paragraph we extend the notion of a skein module to colored bonded links.



Let $R$ be a commutative ring with units $l$ and $m$ an let $l^2\pm ml+1$ and $l^2+1$ also be invertible in $R$.
We denote by $R[\links]$ the free $R$-module spanned by colored bonded links $\links$ and by $\submodule(R,l,m)$ the submodule generated by the following HOMFLYPT skein expression

\begin{equation}\label{eq:skein}
l\icon{5} +l^{-1} \icon{6} +m \icon{7}.\tag{H}
\end{equation}
The HOMFLYPT expression consists of a linear combination of three links in $\links$ that are everywhere the same, except in a small disk in $\R^2$, where they look like the depicted figures.
The \emph{HOMFLYPT skein module of bonded links} is the quotient module 
$$\hsm(R,l,m) = R[\links] / \submodule(R,l,m).$$

Similarly, if we take $\rsubmodule(R,l,m)$ to be the submodule of  $R[\rlinks]$ generated by the HOMFLYPT skein expression~\eqref{eq:skein}, 
the \emph{HOMFLYPT skein module of rigid colored bonded links} is the module 
$$\rhsm(R,l,m) = R[\rlinks] / \rsubmodule(R,l,m).$$

We define the following  \emph{elementary colored bonded knots} with the bond color $i$:
\begin{equation}\label{eq:elementary}
\Theta_i = \imga{i}, \qquad
\bTheta_i = \imgb{i}, \qquad
\bTheta_i^* = \imgB{i}, \qquad
H_i = \imgc{i}, \qquad
\bH_i = \imgd{i}.
\end{equation}

i.e. three oriented $\Theta$-curves and two oriented handcuff links (as non-rigid links $\Theta_i=\bTheta_i=\bTheta_i^*$ and $H_i=\bH_i$). Throughout the paper we denote by $U$ the unknot.


Let $\rbasis$ be the set of all finite unordered products of $\Theta_i$, $\bTheta_i$, $H_i$, and $\bH_i$:
$$\rbasis = \Big\{\prod_{i=1}^k \Theta_i^{m_{i}} \bTheta_i^{m'_{i}}   H_i^{n_{i}} \bH_i^{n'_{i}}  \mid k\in\mathbb N; \; \vec m, \vec m', \vec n, \vec n' \in \mathbb N_0^k\setminus \vec 0 \Big\} \cup \{ \unknot \}$$

and let $\basis$ be the set of all finite products of $\Theta_i$'s:
$$\basis = \Big\{\prod_{i=1}^k \Theta_i^{n_{i}}  \mid k\in\mathbb N; \; \vec n \in \mathbb N_0^k\setminus \vec 0 \Big\} \cup \{\unknot\}.$$

The above products can be regarded as unordered, since the multiplication operation, the \emph{disjoint sum}, is commutative.


We are ready to state our two main theorems, which we will prove in Sections~\ref{sec:rigid} and~\ref{sec:non-rigid}, respectively.

\begin{theorem}\label{thm:main2}
The HOMFLYPT skein module of  colored rigid bonded links, $\rhsm(R,l,m)$, is a free module generated by $\rbasis$.
\end{theorem}

\begin{theorem}\label{thm:main}
The HOMFLYPT skein module of colored bonded links, $\hsm(R,l,m)$, is a free module generated by $\basis$.
\end{theorem}

\begin{remark}\label{remark}
If we define $\pi = \bigcup_{i\in \N} \{ \Ti, \bTi, \Hi, \bHi \}$ (resp. $\pi = \Theta_1 \cup \Theta_2 \cup \cdots$ ), then
$\rhsm(R,l,m)$ (resp. $\hsm(R,l,m)$) is isomorphic to $\mathbf{S} R \pi$, the symmetric tensor algebra over $R \pi$.
To recall, the tensor algebra $\mathbf{T} R \pi$ is the graded sum $\bigoplus_i T^i R \pi$, where
$T^0 R \pi = R$, $T^1 R \pi = R \pi$, $T^{i+1} R \pi = T^{i}R \pi \tensor R \pi$ and the symmetric tensor algebra is the quotient $\mathbf{S} R \pi = \mathbf{T} R \pi / (a \tensor b - b \tensor a)$ (see~\cite{Przytycki2006}).
\end{remark}

The connected sum operation of bonded knots is not unique in general, it is, however, unique in the modules $\hsm(R,l,m)$ and $\rhsm(R,l,m)$ as the following lemma states.

\begin{lemma}In $\hsm(R,l,m)$ and $\rhsm(R,l,m)$ the \emph{connected sum} of two links, $L_1$ and $L_2$ can be expressed by a disjoint sum $L_1L_2$ in terms of the following formula:

\begin{equation*}
\sumpair{16} \;=\, -\frac{m}{l+l^{-1}} \,\sumpair{17}\,.
\end{equation*}\label{lemma:sum}
\end{lemma}
\begin{proof} The formula follows from the following computation:
$$\sumpair{16} \isoteq \sumpair{18}  \Heq -\frac{1}{l^2} \sumpair{19} - \frac{m}{l} \sumpair{20} \isoteq -\frac{1}{l^2} \sumpair{16} -\frac{m}{l} \sumpair{17}.$$

\end{proof}

For a given bond in a diagram, we say that the bond is either \emph{parallel} or \emph{antiparallel}, depending on the relative orientation of the adjacent arcs as depicted in Figure~\ref{fig:parallel}. Note that for rigid bonded knots, this bond property is unchanged under the Reidemeister moves.

\begin{figure}
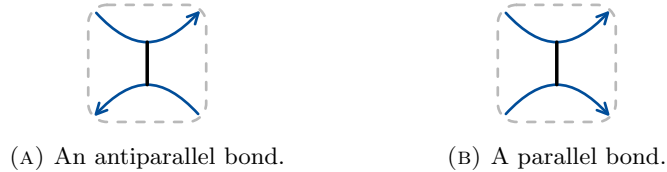

\centering
\subcaptionbox{An antiparallel bond.\label{fig:or1}}[15em]
{\Q{1}}
\subcaptionbox{A parallel bond.\label{fig:or2}}[15em]
{\Q{2}}
\caption{The two possible orientations of adjacent arcs define the bond to be either parallel or antiparallel.}\label{fig:parallel}
\end{figure}



\section{The rigid case} \label{sec:rigid}



Let us note that the statement of Theorem~\ref{thm:main2} collapses to that of singular links in~\cite{Paris2013} if we consider only uncolored parallel bonds and contract them to a point. 
Before proving Theorem~\ref{thm:main2}, we construct a somewhat bigger and simpler structure, which will enable us to canonically choose a bond from a colored bonded link.

Let $\olinks$ the set of links, where the bonds form a tuple $\mathbf{\hat b} = (b_1, b_2, \ldots, b_d)$ and the coloring function is $\hat\chi: (b_1, b_2, \ldots, b_d) \mapsto (c_1,c_2,\ldots,c_d) = c \in \N^d$. For simplicity, we say that the image $c$ is the \emph{coloring} of $L$.
Let $L$ and $L'$ be two ordered bonded links with bonds $\mathbf{\hat b} = (b_1,b_2,\ldots,b_d)$ and $\mathbf{\hat b'} = (b'_1,b'_2,\ldots,b'_{d'})$ colored by $c = (c_1,c_2,\ldots,c_d)$ and $c' = (c'_1,c'_2,\ldots,c'_{d'})$, respectively. The (ordered) disjoint sum $LL'$ forms a link with bonds given by the concatenation  $\mathbf{\hat b}\oplus \mathbf{\hat b'} = (b_1,b_2,\ldots,b_d, b'_1,b'_2,\ldots,b'_{d'})$ and a coloring given by $c\oplus c' = (c_1,c_2,\ldots,c_d, c'_1,c'_2,\ldots,c'_{d'})\in \N^{d+d'}$.

The HOMFLYPT skein module of ordered bonded links, $\ohsm(R,l,m)$, is defined as $R[\olinks]$ quotiented by the skein relations.
The module $\ohsm(R,l,m)$ has a natural grading given by the number of bonds:
$$\ohsm(R,l,m) = \bigoplus_{d=0}^{\infty} \ohsms{d}(R,l,m),$$
where terms in $\ohsms{d}(R,l,m)$ consists of classes of links with $d$ bonds.
Furthermore, each $\ohsms{d}(R,l,m)$ can be graded even further by the coloring:
$$\ohsms{d}(R,l,m) = \bigoplus_{c\in \N^d}\ohsms{d,c}(R,l,m).$$

The next two propositions show that the module $\ohsms{d,c}(R,l,m)$ is freely generated by the set 
$$\obasiss{d,c} = \Big\{ A_{c_1}A_{c_2}\cdots A_{c_d} \mid A_{c_i} \in \{ \Tci, \bTci, \Hci, \bHci \},\, i = 1,\ldots,d  \Big\} \cup \{\unknot\}.$$


\begin{proposition}\label{prop:1}
The module $\ohsms{d,c}(R,l,m)$ is generated by $\obasiss{d,c}$.
\end{proposition}
\begin{proof}
We prove the lemma by induction on $d$. For $d=0$, we do not have any bonds and the module corresponds to the HOMFLYPT skein module of classical links in $S^3$ and is thus freely generated by the unknot $\unknot$ (see~\cite{Przytycki2006, Paris2013}).

Let $c_{-d} = (c_1,c_2,\ldots,c_{d-1})$\footnote{Notation borrowed from game theory.}.
For the induction step we
show that in the module every $d$-component link can be expressed as linear combination of products of elements in $\ohsms{d-1,c_{-d}}(R,l,m)$ and elements in $\hat\pi_{c_d} = \{\Theta_{c_d},\bTheta_{c_d},H_{c_d},\bH_{c_d}\}$. 
In order to show this, we must first isolate the last bond, with index $d$, so that it does not contain any crossings (as in Figure~\ref{fig:or1} and Figure~\ref{fig:or2}).  This can always be achieved using move IV repetitively.

We split the proof into two cases. First, let the bond be antiparallel. 
In this case we have the following equality (as in all subsequent diagrammatic computations, we omit coloring the bonds in all but the first step):

\begin{equation} \label{eq:anti1}
\begin{split}
\Qtext{1}{28}{45}{c_d} & \isoteq \Q{3} \Heq -l^2 \Q{5} -lm\Q{4} \\
& \Heq \Q{7} + lm\Q{6} -lm\Q{4} \\
& \Heq -l^2 \Q{9} - lm \Q{8} + lm\Q{6} -lm\Q{4} \\
& \isoteq -l^2 \Q{1} - lm \Q{12} + lm\Q{11} -lm\Q{10} 
\end{split}
\end{equation}

Terms on the right-hand side are connected sums with an elementary generator in $\hat\pi_{c_d}$, except the first and the third term. We transform the third term:

\begin{equation} \label{eq:anti2}
\begin{split}
\Qtext{11}{43}{57}{c_d}  & \isoteq \Q{13} \Heq -l^{-2} \Q{15}-\frac{m}{l}\Q{14} \\
& \Heq \Q{17} + \frac{m}{l}\Q{16} -\frac{m}{l}\Q{14} \\
& \Heq -l^{-2} \Q{19} - \frac{m}{l} \Q{18} + \frac{m}{l}\Q{16} -\frac{m}{l}\Q{14} \\
& \isoteq -l^{-2} \Q{11} - \frac{m}{l} \Q{21} + \frac{m}{l}\Q{1} -\frac{m}{l}\Q{20} 
\end{split}
\end{equation}

We solve Equations~\eqref{eq:anti1} and~\eqref{eq:anti2} for the diagrams on the left-hand side and obtain the solution:

\begin{equation} \label{eq:antisol1}
\begin{split}
-(l^4 + 2l^2 + 1 - l^2m^2)\Q{1}  = (l^3+l)m&  \Bigg( \Q{10}  + \Q{12} \Bigg) \\
  + l^2m^2  & \Bigg( \Q{21}  + \Q{20} \Bigg) 
\end{split}
\end{equation}
and
\begin{equation} \label{eq:antisol2}
\begin{split}
-(l^4 + 2l^2 + 1 - l^2m^2)\Q{11}  = (l^3+l)m &  \Bigg( \Q{21}  + \Q{20} \Bigg) \\
  + l^2m^2  & \Bigg( \Q{10}  + \Q{12} \Bigg) 
\end{split}
\end{equation}

Note that  $(l^4 + 2l^2 + 1 - l^2m^2) = (l^2+1-lm)(l^2+1+lm)$ is invertible in $R$. 
The right-hand side terms are connected sums with elementary generators from  $\hat\pi_{c_d}$.
We cut out the elementary generators from \eqref{eq:antisol1} using Lemma~\ref{lemma:sum} and obtain:
\begin{equation} \label{eq:antieq1}
(l^4 + 2l^2 + 1 - l^2m^2)\icn{1}  \sumeq l^2m^2  \bigg(   \icn{2} \cdot H_{c_d} + \icn{5} \cdot \Theta_{c_d} \bigg) \\
  + \frac{l^3m^3}{1+l^2}   \bigg(  \icn{2}\cdot \Theta_{c_d}   + \icn{5}\cdot H_{c_d} \bigg).
\end{equation}


From \eqref{eq:antisol2} we obtain an equivalent (symmetric) formula:
\begin{equation} \label{eq:antieq2}
(l^4 + 2l^2 + 1 - l^2m^2)\icn{23}  \sumeq l^2 m^2  \bigg(  \icn{5}\cdot H_{c_d} + \icn{2}\cdot \Theta_{c_d} \bigg) \\
  + \frac{l^3m^3}{1+l^2}     \bigg( \icn{5} \cdot \Theta_{c_d}  +   \icn{2} \cdot H_{c_d}    \bigg). 
\end{equation}

In the second case, when the bond is parallel, we repeat a similar calculation to that of Equation~\eqref{eq:anti1}, except that the bottom arc is reversed:
\begin{equation} \label{eq:par1}
\begin{split}
\Qtext{2}{28}{45}{c_d}  &\isoteq \Q{22}  \Heq  -\frac{1}{l^2} \Q{24}  -\frac{m}{l} \Q{23} \\
                       &\Heq \Q{26}   +\frac{m}{l} \Q{25}  -\frac{m}{l}  \Q{23} \\
       &\Heq -l^2 \Q{28}   -lm \Q{27} +\frac{m}{l} \Q{25}  -\frac{m}{l}  \Q{23} \\
&\isoteq -l^2 \Q{2} - lm \Q{31} + \frac{m}{l} \Q{30}  -\frac{m}{l} \Q{29}
\end{split}
\end{equation}


We isotope the bottom right strand under the rest of the diagram and repeat the calculation:
\begin{equation} \label{eq:par2}
\begin{split}
\Qtext{2}{28}{45}{c_d}  &\isoteq \Q{32}  =  -l^2 \Q{34}  -lm \Q{33} \\
                       &\Heq \Q{36}   +lm \Q{35}  -lm \Q{33} \\
       &\Heq -\frac{1}{l^2}\Q{28}   -\frac{m}{l} \Q{37}  +lm \Q{35} -lm \Q{33} \\
&\isoteq -\frac{1}{l^2} \Q{2} -\frac{m}{l} \Q{31} + lm\Q{40} -lm\Q{39}
\end{split}
\end{equation}

We also have the following HOMFLYPT skein relations:
\begin{equation} \label{eq:he1}
-\frac{m}{l}\Q{29} - lm\Q{39} \Heq m^2\Q{2}
\end{equation}
\begin{equation} \label{eq:he2}
\frac{m}{l}\Q{50} \Heq -lm\Q{44} - m^2\Q{45}
\end{equation}
\begin{equation} \label{eq:he3}
lm\Q{51} \Heq -\frac{m}{l}\Q{44} - m^2\Q{46}
\end{equation}

Adding Equations \eqref{eq:par1} and \eqref{eq:par2} and using Equations \eqref{eq:he1}, \eqref{eq:he2} and \eqref{eq:he3} we get after factorizing
\begin{equation} \label{eq:parsol1}
\begin{split}
(l^2 + l^{-2}+2-m^2)\Q{2}  = -(l+1^{-1})m  & \Bigg(\Q{31} + \Q{44}\Bigg)\\
-m^2 & \Bigg(\Q{45} +  \Q{46} \Bigg).
\end{split}
\end{equation}

Again, the elementary generators from the right-hand sides can be cut off using Lemma~\ref{lemma:sum}. We obtain the formula: 
 \begin{equation} \label{eq:pareq1}
(l^4 + 2l^2 + 1 - l^2m^2)\icn{3}  \sumeq l^2m^2  \bigg(  \icn{4} \cdot \bH_{c_d}  + \icn{7} \cdot \bTheta_{c_d}  \bigg) \\
  + \frac{l^3m^3}{1+l^2}   \bigg(  \icn{4}\cdot \bTheta_{c_d} +  \icn{7}\cdot \bH_{c_d}  \bigg).
\end{equation}

\end{proof}

To obtain a mirror version of Formula \eqref{eq:pareq1} we take the following two HOMFLYPT skein relations:

\begin{equation} \label{eq:he4}
lm\Q{53} \Heq -\frac{m}{l}\Q{47} - m^2\Q{49}
\end{equation}
\begin{equation} \label{eq:he5}
\frac{m}{l}\Q{52} \Heq -lm\Q{47} - m^2\Q{48}
\end{equation}

Similar as before, we add Equations \eqref{eq:par1} and \eqref{eq:par2} and using Equations \eqref{eq:he1}, \eqref{eq:he4} and \eqref{eq:he5} we get:

 \begin{equation} \label{eq:pareq2}
(l^4 + 2l^2 + 1 - l^2m^2)\icn{3}  \sumeq l^2m^2  \bigg(  \icn{4} \cdot \bH_{c_d}  +  \icn{6} \cdot \bTheta_{c_d}^-  \bigg) \\
  + \frac{l^3m^3}{1+l^2}   \bigg(\icn{4}\cdot \bTheta_{c_d}^- +  \icn{6}\cdot \bH_{c_d}   \bigg).
\end{equation}


\begin{proposition}
The module $\ohsms{d,c}(R,l,m)$ is freely generated by $\obasiss{d,c}$.
\end{proposition}
\begin{proof}
We define four $R$-linear maps $g_{d,0}$, $g_{d,\infty}$, $g_{d,+}$, $g_{d,-} : \olinkss{d,c} \rightarrow \olinkss{d-1,c_{-d}}$, which locally replace the last $d$-th bond of each generator with a non-bond according to the following conventions:
\begin{center} 
\begin{tabular}{rlrl}
$g_{d,0}:$ & $\icn{1} \mapsto \icn{2}$ & \qquad\qquad $g_{d,\infty}:$ & $\icn{1} \mapsto \icn{5}$ \\[1em]
& $\icn{3} \mapsto \icn{4}$ &                                    & $\icn{3} \mapsto 0$ \\[2em]
$g_{d,+}:$ & $\icn{1} \mapsto 0$       & \qquad\qquad $g_{d,-}:$ & $\icn{1} \mapsto 0$ \\[1em]
& $\icn{3} \mapsto \icn{6}$ &                                    & $\icn{3} \mapsto \icn{7}$ \\
\end{tabular}
 \end{center}

It is easy to see that the maps  $g_{d,0}$, $g_{d,\infty}$, $g_{d,+}$, and $g_{d,-}$ are well defined (see Theorem 3.1 in~\cite{Kauffman1989}). The most non-trivial case is perhaps invariance of $g_{d,\infty}$ under move RV when the bond is antiparallel. In this case invariance is due to the fact that the
following diagram commutes:

\begin{center}
\begin{tikzcd}
\icn{8} \arrow[r,mapsto, "g_{d,\infty}"] \arrow[d,equal, "\text{V}"] & \icn{10}  \arrow[d,equal, "\text{I,I}"] \\
\icn{9} \arrow[r,mapsto, "g_{d,\infty}"]& \icn{11}
\end{tikzcd}
\end{center}

The four maps can be extended $R$-linearly to maps $g_{d,0}$, $g_{d,\infty}$, $g_{d,+}$, $g_{d,-} : R\olinkss{d,c} \rightarrow R\olinkss{d-1,c_{-d}}$,
which induce the following maps on the module: 
$$g_{d,0}^*, g_{d,\infty}^*, g_{d,+}^*, g_{d,-}^* : \ohsms{d,c}(R,l,m)  \rightarrow \ohsms{d-1,c_{-d}}(R,l,m).$$

On the elementary links $\Theta_{c_d}$,  $\bTheta_{c_d}$, $H_{c_d}$, and $\bH_{c_d}$ maps $g_{1,0}^*$, $g_{1,\infty}^*$, $g_{1,+}^*$, and $g_{1,-}^*$ take the values presented in Table~\ref{tab:vals}. In these calculations we used the following two equalities that hold in $\ohsms{0}(R,l,m)$:
$$\unlinkfig \sumeq \frac{-(l+l^{-1})}{m} \unknotfig \qquad \text{and} \qquad \raisebox{-0.9em}{\includegraphics[page=18]{icons}} \Heq -\frac{1}{l^2} \unlinkfig -\frac{m}{l} \unknotfig \sumeq \Big(\frac{l+l^{-1}}{l^2m} -\frac{m}{l}\Big) \unknotfig.$$

\begin{table}[ht]
\begin{tabular}{|l|c|c|c|c|}\hline
 & $ \smgenth{c_d}$ & \smgenthb{c_d} & \smgenH{c_d} & \smgenHb{c_d} \\[0.25em]\hline
$g_{1,0}^*$       & \unknotfig                         & \unknotfig     & $\frac{-(l+l^{-1})}{m} \unknot$ & $\frac{-(l+l^{-1})}{m} \unknotfig$\\[0.25em]
$g_{1,\infty}^*$  & $\frac{-(l+l^{-1})}{m} \unknotfig$ & 0 & \unknotfig & 0 \\[0.25em]
$g_{1,+}^*$            & 0 & $\frac{l^2-m^2l^2 +1}{l^3m} \unknotfig$ & 0 & \unknotfig \\[0.25em] 
$g_{1,-}^*$     & 0 & $\frac{-(l+l^{-1})}{m} \unknotfig$ & 0 & $\unknotfig$ \\[0.25em]\hline
\end{tabular}  \caption{Values of elementary links $\Theta_{c_d}$, $\bTheta_{c_d}$, $H_{c_d}$, and $\bH_{c_d}$ on maps  $g_{1,0}^*$, $g_{1,\infty}^*$, and $g_{1,\pm}^*$.}\label{tab:vals}
\end{table}

Note that for $B \in \obasiss{d-1,c_{-d}}$ and $b \in \{ \Tci, \bTci, \Hci, \bHci \}$ it holds that $g^*_d (B b) = B g^*_1(b)$.

Let us assume that an arbitrary formal sum in $\ohsms{d,c}(R,l,m)$ takes the value $0$,
$$\sum_{A \in \obasiss{d,c}} \!\!\! r(A) \, A = 0,$$
where $r(A) \in R$. We apply $g_{d,0}^*$ to the sum and obtain

\begin{equation}\label{eq:1a}\tag{12a}\begin{split}
0 &= \kern-1.4em\sum_{B \in \obasiss{d-1,c_{-d}}}\kern-1.4em B\,\Big( 
r(B \Theta_{c_d}) \,g_{1,0}^*(\Theta_{c_d}) +
r(B \bTheta_{c_d}) \,g_{1,0}^*(\bTheta_{c_d}) +
r(B H_{c_d}) \,g_{1,0}^*(H_{c_d}) +
r(B \bH_{c_d}) \,g_{1,0}^*(\bH_{c_d})
\Big)
\\
&= \kern-1.4em\sum_{B \in \obasiss{d-1,c_{-d}}}\kern-1.4em B\,\Big( 
r(B \Theta_{c_d})  +
r(B \bTheta_{c_d})  +
\tfrac{-(l+l^{-1})}{m}  r(B H_{c_d}) +
\tfrac{-(l+l^{-1})}{m}  r(B \bH_{c_d}) 
\Big).
\end{split}\end{equation}

Similarly we apply $g_{d,\infty}^*$ and obtain

\begin{equation}\label{eq:1b}\tag{12b}\begin{split}
0 &= \kern-1.4em\sum_{B \in \obasiss{d-1,c_{-d}}}\kern-1.4em B\,\Big( 
r(B \Theta_{c_d}) \,g_{1,\infty}^*(\Theta_{c_d}) +
r(B \bTheta_{c_d}) \,g_{1,\infty}^*(\bTheta_{c_d}) +
r(B H_{c_d}) \,g_{1,\infty}^*(H_{c_d}) +
r(B \bH_{c_d}) \,g_{1,\infty}^*(\bH_{c_d})
\Big)
\\
&= \kern-1.4em\sum_{B \in \obasiss{d-1,c_{-d}}}\kern-1.4em B\,\Big( 
\tfrac{-(l+l^{-1})}{m} r(B \Theta_{c_d})  +
r(B H_{c_d})
\Big).
\end{split}\end{equation}

By applying the last two maps $g_{d,\pm}^*$ we obtain

\begin{equation}\tag{12c}\begin{split}\label{eq:1c}
0 &= \kern-1.4em\sum_{B \in \obasiss{d-1,c_{-d}}}\kern-1.4em B\,\Big( 
r(B \Theta_{c_d}) \,g_{1,+}^*(\Theta_{c_d}) +
r(B \bTheta_{c_d}) \,g_{1,+}^*(\bTheta_{c_d}) +
r(B H_{c_d}) \,g_{1,+}^*(H_{c_d}) +
r(B \bH_{c_d}) \,g_{1,+}^*(\bH_{c_d})
\Big)
\\
&= \kern-1.4em\sum_{B \in \obasiss{d-1,c_{-d}}}\kern-1.4em B\,\Big( 
\tfrac{l^2-m^2l^2 +1}{l^3m} r(B \bTheta_{c_d})  +
r(B \bH_{c_d})
\Big)
\end{split}\end{equation}

and

\begin{equation}\tag{12d}\begin{split}\label{eq:1d}
0 &= \kern-1.4em\sum_{B \in \obasiss{d-1,c_{-d}}}\kern-1.4em B\,\Big( 
r(B \Theta_{c_d}) \,g_{1,-}^*(\Theta_{c_d}) +
r(B \bTheta_{c_d}) \,g_{1,-}^*(\bTheta_{c_d}) +
r(B H_{c_d}) \,g_{1,-}^*(H_{c_d}) +
r(B \bH_{c_d}) \,g_{1,-}^*(\bH_{c_d})
\Big)
\\
&= \kern-1.4em\sum_{B \in \obasiss{d-1,c_{-d}}}\kern-1.4em B\,\Big( 
\tfrac{-(l+l^{-1})}{m} r(B \bTheta_{c_d})  +
r(\bH_{c_d}B)
\Big).
\end{split}\end{equation}

By the induction hypothesis Equations \eqref{eq:1a}\,--\,\eqref{eq:1d} yield the following linear system of equations:

\begin{alignat*}{5}
 r(B\Theta_{c_d}) & {}+{} & r(B\bTheta_{c_d}) & {}-{} &\tfrac{(l+l^{-1})}{m}  r(BH_{c_d}) & {}-{} &\tfrac{(l+l^{-1})}{m} r(BH_{c_d}) & {}={} & 0 \\
-\tfrac{(l+l^{-1})}{m}  r(B \Theta_{c_d}) & {}+{} &  &  &  r(BH_{c_d})    & & & {}={} &  0 \\
 &  & \tfrac{l^2-m^2l^2 +1}{l^3m} r(B\bTheta_{c_d}) &  &   & {}+{} & r(BH_{c_d})  & {}={} & 0 \\
 &{}-{}  & \tfrac{(l+l^{-1})}{m}  r(B\bTheta_{c_d}) &  &   & {}+{} & r(BH_{c_d})  & {}={} & 0 
\end{alignat*}

The determinant of this system is $l^{-5} m^{-3} (l^2 + l m + 1)^2 ( l^2 - l m + 1)^2$, which is invertible in $R$.
It holds $r(B\Theta_{c_d}) = r(B\bTheta_{c_d}) = r(BH_{c_d}) = r(B\bH_{c_d}) = 0$ for all $B \in \obasiss{d-1,c_{-d}}$.

\end{proof}

\begin{proof}[Proof of theorem \ref{thm:main2}] 

As in the case of $\ohsm(R,l,m)$, the module $\rhsm(R,l,m)$ also has a natural grading given by the number of bonds $d$: $$\rhsm(R,l,m) = \bigoplus_{d=1}^\infty \rhsms{d}(R,l,m).$$ 
In order to prove the theorem, it is enough to show that $\rhsms{d}(R,l,m)$ is freely generated by the set 
$$\rbasiss{d} = \bigg\{\prod_{i=1}^k \Theta_i^{m_{i}} \bTheta_i^{m'_{i}} H_i^{n_{i}} \bH_i^{n'_{i}} \mid k\in\mathbb N; \;\vec m, \vec m', \vec n, \vec n' \in \mathbb N_0^k\setminus \vec 0,\,  \sum_{i=1}^k \big(m_i+m'_i+n_i+n'_i\big) = d \bigg\} \cup \{ \unknot \},$$
consisting of unordered $d$-products of  colored elementary generators.

The symmetric group $S_d$ acts on the coloring $c=(c_1,c_2,\ldots,c_d)$ by permuting the coordinates. This action also permutes elements in the basis $\obasiss{d}$ and the module $\ohsms{d}(R,l,m) / S_d$ is freely generated by the orbits of $S_n$, which coincide with $\rbasiss{d}$. It follows that $\rhsms{d}(R,l,m) \cong \ohsms{d}(R,l,m) / S_d$, that is, $\rhsms{d}(R,l,m)$ is freely generated by $\rbasiss{d}$.
\end{proof}


\section{The non-rigid case}\label{sec:non-rigid}

\begin{proof}[Proof of theorem \ref{thm:main}] 

The module $\hsm_d(R,l,m)$ by construction equals to $\rhsms{d}(R,l,m)$ in which we impose the additional isotopy generating move V.
In $\hsm_d(R,l,m)$, move $V$ clearly gives us equalities 
\begin{equation}\label{eq:hsmgens}
\Theta_i = \bTheta_i \qquad \text{and} \qquad  H_i = \bH_i = -\frac{1+l ^2}{lm} \Theta_i,
\end{equation}
the latter equality arises from the HOMFLYPT skein relation $\bTheta_i^* = -l^2 \bTheta -lm \bH_i.$
We are left to show invariance under move V.


For the antiparallel version of the right-handed move V we have:
\begin{equation*}
\begin{split}
(l^4 + 2l^2 + 1 - l^2m^2)\icn{25}  
& \eqeq{eq:pareq1} l^2m^2  \bigg(  \icm{27} \cdot \bH_i +  \icm{26} \cdot \bTheta_i  \bigg)  + \frac{l^3m^3}{1+l^2}   \bigg( \icm{26}\cdot \bH_i  + \icm{27}\cdot \bTheta_i   \bigg) \\
&= l^2m^2  \bigg(     \icm{28} \cdot H_i + \icm{29} \cdot \Theta_i\bigg)  + \frac{l^3m^3}{1+l^2}   \bigg( \icm{29}\cdot H_i  +\icm{28}\cdot \Theta_i  \bigg) \\
& \eqeq{eq:antieq1} (l^4 + 2l^2 + 1 - l^2m^2)\icm{30}.  
\end{split}
\end{equation*}

Similarly, for the parallel version of the right-handed move V we have:

\begin{equation*}
\begin{split}
(l^4 + 2l^2 + 1 - l^2m^2)\icn{31}  
& \eqeq{eq:antieq1} l^2m^2  \bigg(   \icm{33} \cdot \bH_i +  \icm{32} \cdot \bTheta_i \bigg)  + \frac{l^3m^3}{1+l^2}   \bigg( \icm{32}\cdot \bH_i +\icm{33}\cdot \bTheta_i   \bigg) \\
&= l^2m^2  \bigg(   \icm{34} \cdot H_i +  \icm{32} \cdot \Theta_i \bigg)  + \frac{l^3m^3}{1+l^2}   \bigg( \icm{32}\cdot H_i  + \icm{34}\cdot \Theta_i  \bigg) \\
& \eqeq{eq:pareq2} (l^4 + 2l^2 + 1 - l^2m^2)\icm{36}.  
\end{split}
\end{equation*}

The left-handed twist is treated analogously.

\end{proof}

Unfortunately, the HOMFLYPT skein module does not recover any information about the knottedness of the bonds, which is evident from the evaluation of the right-hand side of formula \eqref{eq:antieq1} in $\hsm_d(R,l,m)$:

\begin{equation} \label{eq:antieq1HSM}
\begin{split}
\icn{1}  & \eqeq{eq:antieq1}  \frac{1}{l^4 + 2l^2 + 1 - l^2m^2} \Bigg( l^2m^2  \bigg(   \icn{2} \cdot H_{c_d} + \icn{5} \cdot \Theta_{c_d} \bigg)  + \frac{l^3m^3}{1+l^2}   \bigg(  \icn{2}\cdot \Theta_{c_d}   + \icn{5}\cdot H_{c_d} \bigg) \Bigg) \\
& \eqeq{eq:hsmgens}
\frac{1}{l^4 + 2l^2 + 1 - l^2m^2} \Bigg( l^2m^2  \bigg(  -\frac{1+l ^2}{lm} \;  \icn{2}  + \icn{5}  \bigg)  + \frac{l^3m^3}{1+l^2}   \bigg(  \icn{2}   -\frac{1+l ^2}{lm} \;  \icn{5} \bigg) \Bigg)\cdot \Theta_{c_d}\\
& = -\frac{lm}{1+l^2} \; \icn{2} \cdot \Theta_{c_d}.
\end{split}
\end{equation}

However, if we   appropriately modify the coloring function, we can still extract some information about the bonds as we will see in the next section.







\section{Examples}\label{sec:refined}

In practice, given a link $L$ with $d$ bonds, one can compute the invariant $[L]_\rbasis$, the rigid HOMFLYPT skein module of the class of $L$ in the basis $\rbasis$, by the following set of instructions:
\begin{enumerate}
\item isolate the bonds using move IV,
\item cut out the bonds using Equation~\eqref{eq:antieq1} or Equation~\eqref{eq:pareq1} 
\item compute the HOMFLYPT polynomial $P$ of the remaining part of the classical knot.
\end{enumerate}

We can compute $[L]_\basis$ by removing the bonds using Equation~\eqref{eq:antieq1HSM}.

\begin{example}
We demonstrate the usefulness of the invariants by first showing that it can distinguish bonded structures better than regular $\Theta$-curve analysis.
In~\cite{Dabrowski-Tumanski2019} it is shown that the $\Theta$-curve $\Theta 3_1$ in Figure~\ref{fig:ex1} is the most common non-trivial prime $\Theta$-curve appearing in knotted proteins with bonds.
There are three possible ways we can color edges of a $\Theta$-curve to obtain a bonded knot. For the curve $\Theta 3_1$ we denote the associated bonded knots by $A$, $B$ and $C$ as depicted in Figure~\ref{fig:ex2}. 
While it is an easy exercise to check that bonded knots $B$ and $C$ are isotopic, $A$ and $B$ (or $A$ and $C$) are not.

\begin{figure}
\centering
\subcaptionbox{The $\Theta$-curve $\Theta 3_1$.\label{fig:ex1}}[15em]
{\quad\includegraphics[page=16]{example}\quad}
\subcaptionbox{Corresponding bonded trefoils.\label{fig:ex2}}
{    
\quad\begin{overpic}[page=15]{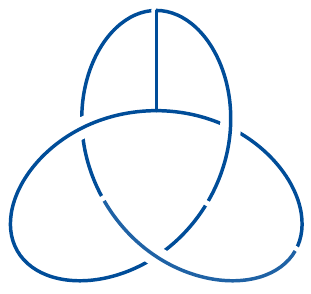}\put(46,49){$A$}\put(17,2){$B$}\put(75,2){$C$}
\end{overpic}\quad
}
\caption{We can associate three bonded knots to a $\Theta$-curve by coloring each of the edges.\label{fig:ex}
}
\end{figure}

Let us first compute $[A]_{\rbasis}$.
As the bond in knot $A$ does not contain any crossings, we do not need to isolate it. Observe in equation \eqref{eq:pareq1} that if we remove the bond, we obtain the classical trefoil knot and if we replace it with a negative crossing, we obtain the Hopf link. In addition, we need to multiply the expression by $\frac{-l-l^{-1}}{m}$ to reduce the additional unknot (Lemma~\ref{lemma:sum}):

\begin{equation*} 
\begin{split}
\rHSMknot{\icex{8}}  \eqeq{eq:pareq1} &
 \frac{1}{l^4+2l^2+1-l^2m ^2}\cdot \frac{-l-l^{-1}}{m}  \\ 
 & \cdot \Bigg(
		\bigg( l^2 m^2 \,\bH + \frac{l^3 m^3}{1+l^2} \, \bTheta \bigg) \, \hpoly{\icex{9}}
+		\bigg( l^2 m^2 \,\bTheta + \frac{l^3 m^3}{1+l^2} \,\bH \bigg) \, \hpoly{\icex{14}}
\Bigg ) \\
= &\,
 \frac{-lm(1+l^2)}{l^4+2l^2+1-l^2m ^2} \\ 
 & \cdot \Bigg(
		\bigg(  \bH + \frac{l m}{1+l^2} \, \bTheta \bigg) \bigg(   \frac{m^2}{l^2} - \frac{2}{l^2} - \frac{1}{l^4} \bigg)
+		\bigg( \bTheta + \frac{l m}{1+l^2} \,\bH \bigg)   \bigg(    \frac{1}{lm} - \frac{m}{l} + \frac{1}{l^3m}  \bigg) 
\Bigg ) \\
  = & \,
(1^{-2}m^2 -l^{-2})) \, \smbT +		 l^{-3}m \,\smbH  
\end{split}
\end{equation*}

\begin{figure}
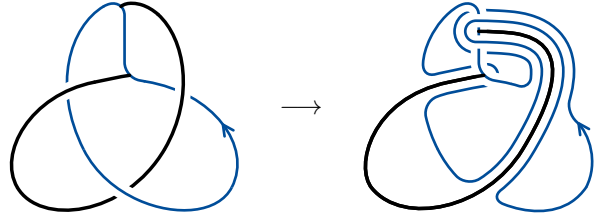

\centering
\includegraphics[page=3]{example} \raisebox{4em}{$\quad \longrightarrow \quad$}
\includegraphics[page=4]{example}
\caption{Isolating a bond.}\label{fig:trefisol}
\end{figure}

On the other hand, as the bond of knot $B$ contains crossings, we isolate it using move IV as depicted in Figure~\ref{fig:trefisol}. By removing the bond we obtain the unknot (cf. Figure~\ref{fig:ex2}) and if we replace it with a crossing, we obtain the link L7n1 in the Thistlethwaite link table, which gives us:

\begin{equation*} 
\begin{split}
\rHSMknot{\icex{11}}  \eqeq{eq:pareq1} &
 \frac{1}{l^4+2l^2+1-l^2m ^2} \cdot \frac{-l-l^{-1}}{m}  
  \cdot \Bigg(
		\bigg( l^2 m^2 \,\bH + \frac{l^3 m^3}{1+l^2} \, \bTheta \bigg) \, \hpoly{\icex{12}} \\ &
+		\bigg( l^2 m^2 \,\bTheta + \frac{l^3 m^3}{1+l^2} \,\bH \bigg) \, \hpoly{\icex{13}} \bigg)  
\Bigg ) \\
= &\,
 \frac{-lm(1+l^2)}{l^4+2l^2+1-l^2m ^2} 
 \cdot \Bigg(
		\bigg(  \bH + \frac{l m}{1+l^2} \, \bTheta \bigg)\\ &
+		\bigg( \bTheta + \frac{l m}{1+l^2} \,\bH \bigg)   \bigg(   -l^3 m + \frac{2l^3}{m} + l m^3    - 4 l m + \frac{3l}{m} - \frac{m}{l} +  \frac{1}{lm} \bigg) 
\Bigg ) \\
  = & \,
(l^2m^2 - 2l^2 + m^2 - 1)\, \smbT + (l m^3 - 2lm) \, \smbH 
\end{split}
\end{equation*}

Viewing the $A$ and $B$ as non-rigid bonded knots, we have:
\begin{equation*} 
\begin{split}
\HSMknot{\icex{8}}  \eqeq{eq:antieq1HSM} &
 -\frac{lm}{1+l^2} \cdot \frac{-l-l^{-1}}{m} \,  \hpoly{\icex{9}} \, \Theta =  (l^{-2}m^2 - 2l^{-2} - l^{-4}) \, \Theta \\
 \HSMknot{\icex{11}}  \eqeq{eq:antieq1HSM} &
 -\frac{lm}{1+l^2} \cdot \frac{-l-l^{-1}}{m} \,  \hpoly{\icex{12}} \, \Theta =  \Theta
\end{split}
\end{equation*}


We conclude that $A \neq B$ viewed either as rigid or non-rigid bonded knots.
\end{example}

In order to better distinguish between non-rigid bonded knots, we can add extra information about the bonds by modifying its coloring function.
For example, we can capture information about the protein's circuit topology (arrangement of its intra-molecular contacts~\cite{Mashaghi2014}) by coloring the bonds with its contact-distance, i.e. minimal number of bond-contacts met while travelling along the knot (see example in Figure~\ref{fig:ct}).

\begin{figure}
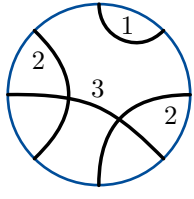

\centering
\begin{overpic}[page=17]{example}\put(61,81){$1$}\put(16,63){$2$}\put(83,35){$2$}\put(46,49){$3$}
\end{overpic}
\caption{Capturing protein's circuit topology by coloring the bonds.\label{fig:ct}
}
\end{figure}

\begin{example}
It is known that toxins from venomous organisms form disulfide-rich peptides and therefore provide good examples of bonded knots.
We compute the refined HOMFLYPT expression of two protein complexes: the toxin CN29 of the Mexican Nayarit Scorpion venom in Figure~\ref{fig:toxin1} and the toxin ADWX-1 of the Chinese scorpion venom in Figure~\ref{fig:toxin2}, where we close the open chains by a direct segment.

\begin{figure}
\centering
\subcaptionbox{CN29 toxin.\label{fig:toxin1}}[15em]
{\includegraphics[page=3]{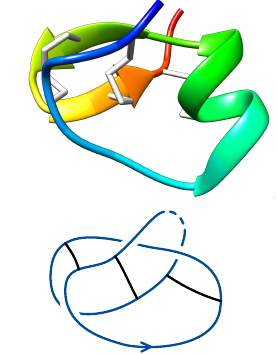}}
\subcaptionbox{ADWX-1 toxin.\label{fig:toxin2}}[15em]
{\includegraphics[page=4]{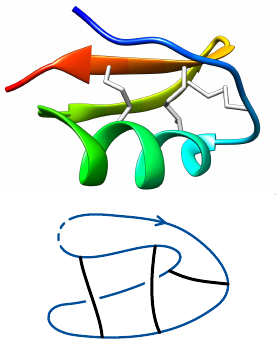}}
\caption{Ribbon diagrams and the associated bonded knots of the toxins from Centruroides noxius (PDB entry 6NW8) (left)  and  Mesobuthus martensii (PDB entry 2K4U) (right).\label{fig:toxin}
}
\end{figure}

The toxin topologies are indeed different as rigid bonded knots in the given projection:

\begin{equation*}
\begin{split}
 \big[K_\text{CN29}\big]_{\rbasis} =\,& \frac{1}{(1 + l^2)^2 (l^4 +2l^2 + l-l^2m^2)^2} \Big ( 
 l^6 m^4(-1 - 3 l^2  - 3 l^4 - l^6 + l^2 m^2 + 2 l^4 m^2) \Ttt \\
&  +  l^5 m^3(1 + 3 l^2  + 3 l^4  + l^6  - m^2 - 6 l^2 m^2 - 6 l^4 m^2 - l^6 m^2 + l^2 m^4 + 3 l^4 m^4) \Tth \\
& + l^7 m^5(-1 - l^2  + l^2 m^2) \ttH + l^6 m^6 (-1 - 2 l^2  + l^2 m^2) \Thh  \\
& + l^6 m^4(-1 - 2 l^2 - l^4 - m^2 - l^2 m^2 + l^4 m^2 + l^2 m^4) \tHh \\
& + l^5 m^5(-1 - 3 l^2  - 2 l^4  + l^2 m^2 + l^4 m^2) \Hhh \Big )
\end{split}
\end{equation*}

\begin{equation*}
\begin{split}
 \big[K_\text{ADWX-1}\big]_{\rbasis} =\,& \frac{1}{(1 + l^2)^2 (l^4 +2l^2 + l-l^2m^2)^2} \Big ( 
l^6 m^4(-1 - 2 l^2 - l^4  + l^4 m^2) \Ttt \\
&+ l^7 m^5 (-4  - 4 l^2 + 2 l^2 m^2) \Tth 
+ l^7 m^5(-1 + l^4) \ttH  \\
&+ l^4 m^4 (1 + 2 l^2  + l^4 - 2 l^2 m^2 - 3 l^4 m^2 + l^4 m^4) \Thh \\
&+ l^6 m^4(-2  - 4 l^2 - 2 l^4 + 2 l^4 m^2) \tHh 
+ l^7 m^5(-2  - 2 l^2  + l^2 m^2) \Hhh \Big )
\end{split}
\end{equation*}

For the non-rigid versions we use the contact-distance coloring function. The topologies of the proteins are different:

\begin{equation*}
\big[K_\text{CN29}\big]_{\basis} = \frac{l^2m^2}{(1 + l^2)^2} \, \Theta_2^2 \Theta_3,
\end{equation*}
\begin{equation*}
\big[K_\text{ADWX-1}\big]_{\basis} = \frac{l^2m^2}{(1 + l^2)^2} \, \Theta_3^3.
\end{equation*}


\end{example}
\subsection*{Acknowledgements} 
The author would like to thank Dimos Goundaroulis and J\'{o}zef Przytycki for fruitful conversations on the topic.
The author was financially supported by the Slovenian Research Agency grants BI-US/19-21-111, N1-0083, N1-0064, and J1-8131.




\bibliographystyle{abbrv}
\bibliography{biblio}

\begin{thebibliography}{10}

\bibitem{Alexander2017}
K.~Alexander, A.~J. Taylor, and M.~R. Dennis.
\newblock Proteins analysed as virtual knots.
\newblock {\em Scientific Reports}, 7(1), 2017.

\bibitem{Chlouveraki2019}
M.~Chlouveraki, D.~Goundaroulis, A.~Kontogeorgis, and S.~Lambropoulou.
\newblock A generalized skein relation for khovanov homology and a
  categorification of the $\theta$-invariant.
\newblock {\em arXiv:1904.07794 [math.GT]}, 2019.

\bibitem{Dabrowski-Tumanski2019}
P.~Dabrowski-Tumanski, D.~Goundaroulis, A.~Stasiak, and J.~I. Sulkowska.
\newblock $\theta$-curves in proteins.
\newblock {\em arXiv:1908.05919 [cond-mat.soft]}, 2019.

\bibitem{Diamantis2016a}
I.~Diamantis and S.~Lambropoulou.
\newblock A new basis for the homflypt skein module of the solid torus.
\newblock {\em Journal of Pure and Applied Algebra}, 220(2):577--605, 2016.

\bibitem{Diamantis2019}
I.~Diamantis and S.~Lambropoulou.
\newblock An important step for the computation of the {HOMFLYPT} skein module
  of the lens spaces $l(p,1)$ via braids.
\newblock {\em Journal of Knot Theory and Its Ramifications}, page 1940007,
  2019.

\bibitem{Diamantis2016}
I.~Diamantis, S.~Lambropoulou, and J.~H. Przytycki.
\newblock Topological steps toward the {H}omflypt skein module of the lens
  spaces $l(p,1)$ via braids.
\newblock {\em Journal of Knot Theory and Its Ramifications}, 25(14):1650084,
  2016.

\bibitem{Freyd1985}
P.~Freyd, D.~Yetter, J.~Hoste, W.~B.~R. Lickorish, K.~Millett, and A.~Ocneanu.
\newblock A new polynomial invariant of knots and links.
\newblock {\em Bulletin of the American Mathematical Society}, 12(2):239--247,
  1985.

\bibitem{Gabrovsek2014}
B.~Gabrov\v{s}ek and M.~Mroczkowski.
\newblock The {HOMFLYPT} skein module of the lens spaces $l_{p,1}$.
\newblock {\em Topology and its Applications}, 175:72--80, 2014.

\bibitem{Gueguemcue2017a}
N.~Gügümcü and L.~H. Kauffman.
\newblock New invariants of knotoids.
\newblock {\em European Journal of Combinatorics}, 65:186--229, 2017.

\bibitem{Gueguemcue2017}
N.~Gügümcü and S.~Lambropoulou.
\newblock Knotoids, braidoids and applications.
\newblock {\em Symmetry}, 9(12):315, 2017.

\bibitem{Goundaroulis2017a}
D.~Goundaroulis, J.~Dorier, F.~Benedetti, and A.~Stasiak.
\newblock Studies of global and local entanglements of individual protein
  chains using the concept of knotoids.
\newblock {\em Scientific Reports}, 7(1), 2017.

\bibitem{Goundaroulis2017}
D.~Goundaroulis, N.~Gügümcü, S.~Lambropoulou, J.~Dorier, A.~Stasiak, and
  L.~Kauffman.
\newblock Topological models for open-knotted protein chains using the concepts
  of knotoids and bonded knotoids.
\newblock {\em Polymers}, 9(12):444, 2017.

\bibitem{Jamroz2014}
M.~Jamroz, W.~Niemyska, E.~J. Rawdon, A.~Stasiak, K.~C. Millett,
  P.~Su{\l}kowski, and J.~I. Sulkowska.
\newblock {KnotProt}: a database of proteins with knots and slipknots.
\newblock {\em Nucleic Acids Research}, 43(D1):D306--D314, 2014.

\bibitem{Kauffman1993}
L.~Kauffman, J.~Simon, K.~Wolcott, and P.~Zhao.
\newblock Invariants of theta-curves and other graphs in 3-space.
\newblock {\em Topology and its Applications}, 49(3):193--216, 1993.

\bibitem{Kauffman1989}
L.~H. Kauffman.
\newblock Invariants of graphs in three-space.
\newblock {\em Transactions of the American Mathematical Society},
  311(2):697--697, 1989.

\bibitem{Liang1994}
C.~Liang and K.~Mislow.
\newblock Knots in proteins.
\newblock {\em Journal of the American Chemical Society}, 116(24):11189--11190,
  1994.

\bibitem{Mansfield1994}
M.~L. Mansfield.
\newblock Are there knots in proteins?
\newblock {\em Nature structural biology}, 1:213--214, 1994.

\bibitem{Mashaghi2014}
A.~Mashaghi, R.~J. van Wijk, and S.~J. Tans.
\newblock Circuit topology of proteins and nucleic acids.
\newblock {\em Structure}, 22(9):1227--1237, sep 2014.

\bibitem{Millett2013}
K.~C. Millett, E.~J. Rawdon, A.~Stasiak, and J.~I. Su{\l}kowska.
\newblock Identifying knots in proteins.
\newblock {\em Biochemical Society Transactions}, 41(2):533--537, 2013.

\bibitem{Mishra2011}
R.~Mishra and S.~Bhushan.
\newblock Knot theory in understanding proteins.
\newblock {\em Journal of Mathematical Biology}, 65(6-7):1187--1213, 2011.

\bibitem{Mroczkowski2004}
M.~Mroczkowski.
\newblock Polynomial invariants of links in the projective space.
\newblock {\em Fundamenta Mathematicae}, 184:223--267, 2004.

\bibitem{Mroczkowski2018}
M.~Mroczkowski.
\newblock The dubrovnik and kauffman skein modules of the lens spaces lp,1.
\newblock {\em Journal of Knot Theory and Its Ramifications}, 27(03):1840004,
  2018.

\bibitem{ODonnol2018}
D.~O'Donnol, A.~Stasiak, and D.~Buck.
\newblock Two convergent pathways of {DNA} knotting in replicating {DNA}
  molecules as revealed by {$\Theta$}-curve analysis.
\newblock {\em Nucleic Acids Research}, 46(17):9181--9188, 2018.

\bibitem{Paris2013}
L.~Paris and E.~Wagner.
\newblock {HOMFLY}-{PT} skein module of singular links in the three-sphere.
\newblock {\em Journal of Knot Theory and Its Ramifications}, 22(02):1350005,
  2013.

\bibitem{Przytycki1991}
J.~H. Przytycki.
\newblock Skein modules of 3-manifolds.
\newblock {\em Bulletin of the Polish Academy of Sciences}, 39(1-2):91--100,
  1991.

\bibitem{Przytycki1992}
J.~H. Przytycki.
\newblock {\em Skein module of links in a handlebody, {T}opology '90:
  {P}roceedings of the {R}esearch Semester in Low Dimensional Topology at
  {OSU}}.
\newblock De Gruyter, 1992.
\newblock Editors: Boris N. Apanasov and Walter D. Neumann and Alan W. Reid and
  Laurent Siebenmann.

\bibitem{Przytycki2006}
J.~H. Przytycki.
\newblock Skein modules.
\newblock {\em arXiv:math/0602264 [math.GT]}, 2006.

\bibitem{Sulkowska2012}
J.~I. Sulkowska, E.~J. Rawdon, K.~C. Millett, J.~N. Onuchic, and A.~Stasiak.
\newblock Conservation of complex knotting and slipknotting patterns in
  proteins.
\newblock {\em Proceedings of the National Academy of Sciences},
  109(26):E1715--E1723, 2012.

\bibitem{Sulkowska2008}
J.~I. Sulkowska, P.~Sulkowski, P.~Szymczak, and M.~Cieplak.
\newblock Stabilizing effect of knots on proteins.
\newblock {\em Proceedings of the National Academy of Sciences},
  105(50):19714--19719, 2008.

\bibitem{Tian2017}
W.~Tian, X.~Lei, L.~H. Kauffman, and J.~Liang.
\newblock A knot polynomial invariant for analysis of topology of {RNA} stems
  and protein disulfide bonds.
\newblock {\em Computational and Mathematical Biophysics}, 5(1), 2017.

\bibitem{Tropp2012}
B.~Tropp.
\newblock {\em Molecular biology: genes to proteins}.
\newblock Jones \& Bartlett Learning, Sudbury, Mass, 2012.

\bibitem{Turaev2012}
V.~Turaev.
\newblock Knotoids.
\newblock {\em Osaka Journal of Mathematics}, 49(1):195--223, 2012.

\bibitem{Turaev1990}
V.~G. Turaev.
\newblock Conway and kauffman modules of a solid torus.
\newblock {\em Journal of Soviet Mathematics}, 52(1):2799--2805, 1990.

\bibitem{Virnau2006}
P.~Virnau, L.~A. Mirny, and M.~Kardar.
\newblock Intricate knots in proteins: Function and evolution.
\newblock {\em {PLoS} Computational Biology}, 2(9):e122, 2006.

\end{thebibliography}

\end{document}